\documentclass{daj}

\usepackage{mathrsfs,amsfonts,amsmath,amssymb}
\usepackage{latexsym}
\newtheorem{satz}{Theorem}
\newtheorem{proposition}[satz]{Proposition}
\newtheorem{theorem}[satz]{Theorem}
\newtheorem{lemma}[satz]{Lemma}
\newtheorem{definition}[satz]{Definition}
\newtheorem{corollary}[satz]{Corollary}
\newtheorem{remark}[satz]{Remark}
\newtheorem{example}[satz]{Example}

\def\eps{\varepsilon}
\def\_phi{\varphi}

\def\a{\alpha}
\def\d{\delta}
\def\la{\lambda}

\def\F{{\mathbb F}}

\def\t{\tilde}

\def\C{{\mathbb C}}
\def\R{{\mathbb R}}
\def\E{\mathsf {E}}
\def\T{{\mathbb T}}

\def\Z_N{{\mathbb Z}_N}
\def\Z{{\mathbb Z}}
\def\N{{\mathbb N}}

\def\f{{\mathbb F}}
\def\Gr{{\mathbf G}}

\def\D{{\mathbb D}}

\def\G{\Gamma}

\def\D{\Delta}

\def\T{\mathsf {T}}

\dajAUTHORdetails{%
  title = {Difference Sets are Not Multiplicatively Closed}, 
  author = {Shkredov I.D.},
  plaintextauthor = {Shkredov I.D.},
  keywords = {sumsets, additive energy},
}   

\dajEDITORdetails{%
   year={2016},
   number={17},
   received={16 February 2016},   
   published={5 October 2016},  
   doi={10.19086/da.913},       
}   

\begin{document}

\begin{frontmatter}[classification=text]

\title{Difference Sets are Not Multiplicatively Closed} 

\author[ilya]{Shkredov, I. D.}

\begin{abstract}
We prove that for any finite set $A\subset \R$ its difference set $D:=A-A$ has large product set and quotient set: that is, $|DD|, |D/D| \gg |D|^{1+c}$, where $c>0$ is an absolute constant.
    A similar result takes place in the prime field
    $\F_p$ for sufficiently small $D$.
    It gives, in particular, that
    multiplicative subgroups of size less than $p^{4/5-\eps}$
    cannot be represented in the form $A-A$ for any $A\subset \F_p$.
\end{abstract}
\end{frontmatter}

\section{Introduction}
\label{sec:introduction}

Let  $A,B\subset \R$ be finite sets.
Define the  \textit{sum set}, the \textit{difference set},
the \textit{product set} and the \textit{quotient set} of $A$ and $B$ to be
$$A+B:=\{a+b ~:~ a\in{A},\,b\in{B}\}\,, \quad \quad \quad  A-B:=\{a-b ~:~ a\in{A},\,b\in{B}\}$$
$$AB:=\{ab ~:~ a\in{A},\,b\in{B}\}\,,\quad  \quad \quad A/B:=\{a/b ~:~ a\in{A},\,b\in{B},\,b\neq0\}\,,$$
respectively.
The Erd\H{o}s--Szemer\'{e}di  sum--product conjecture \cite{ES} says that for any  $\epsilon>0$ one has
$$\max{\{|A+A|,|AA|\}}\gg{|A|^{2-\epsilon}} \,.$$
Thus, it asserts that for an arbitrary subset of real numbers (or integers) either the sumset or the product set of this set is large.
Modern bounds concerning the conjecture can be found in \cite{soly}, \cite{KS_smd}, \cite{KS2}.

We consider the following sum--product type question.

{\bf Problem.}
{\it Let $A,P \subset \R$ be two finite sets, $P \subseteq A-A$.
Suppose that $|PP| \le |P|^{1+\eps}$, where $\eps >0$ is a small parameter.
In other words,  $P$ has small product set.
Is it true that there exists $\d = \d(\eps) > 0$ with
\begin{equation}\label{f:problem_A_P}
    \sum_{x\in P} |\{ a_1-a_2 = x ~:~ a_1,a_2 \in A \}| \ll |A|^{2-\d}\ ?
\end{equation}
}

Thus, we consider a set $P$ with small product set
and we want to say something nontrivial about the additive structure of $A$, that is, the sum of additive convolutions
over $P$.
The question
plays an important role in recent papers \cite{RN_Z}, \cite{Solodkova_S} and \cite{Z_progr}, \cite{Z_sphere}.
Even the famous unit distance problem of  Erd\H{o}s (see \cite{Erdos_dist} and also the survey \cite{Sol_subspace})
can be considered as a question of that
type.
Indeed,  the
unit distance problem
is
to find a  good
upper bound for
$$
    \int_{S^{1}} |\{ a_1-a_2 = x ~:~ a_1,a_2 \in A \}| \,\, dx  \,,
$$
where $A$ is a finite subset of the Euclidean plane (which we consider as the complex plane)
and $S^{1}$ is the unit circle. Since the unit circle
$S^{1}$
is
a subgroup, it trivially has small product set: $S^{1} \cdot S^{1} = S^{1}$ .

Let us return to (\ref{f:problem_A_P}).
In \cite{RN_Z} Roche--Newton and Zhelezov
studied some
sum--product type
questions
and
obtained the following
result
(in principle, the method of their paper allows one to obtain subexponential bounds for the sum from (\ref{f:problem_A_P}) but not of the required form).
The multiplicative energy of $A$, denoted by $\E^\times (A)$, is the number of solutions to the equation
$ab = cd$, where $a,b,c,d \in A$.

\begin{theorem}
    For any $\eps>0$ there are constants $C' (\eps), C''(\eps)>0$ such that for any set $A\subset \C$
    one has
\begin{equation}\label{f:RN_E*}
    \E^\times (A-A) \le \max\{ C'' (\eps) |A|^{3+\eps},\, |A-A|^{3} \exp(-C'(\eps) \log^{1/3-o(1)} |A|) \} \,.
\end{equation}
\label{t:RN_E*}
\end{theorem}

Thus, bound (\ref{f:RN_E*}) says us that the difference set $D=A-A$ enjoys  a non--trivial upper bound for
its
multiplicative energy.
The proof used a deep result of Sanders \cite{Sanders_Fr} and the Subspace Theorem of Schmidt (see, e.g., \cite{Sol_subspace}).
Roughly speaking, thanks to
the Subspace Theorem,
Roche--Newton and Zhelezov
obtained an upper bound for the sums from (\ref{f:problem_A_P}) for $P$ equal to a multiplicative subgroup of $\C$ of small rank (so $P$ automatically has small product set for trivial reasons)
and using Sanders' structural result they extended it to general sets with small multiplicative doubling -- see details in \cite{RN_Z}.

\bigskip

Theorem \ref{t:RN_E*} has the following consequence.

\begin{corollary}
    Let $A\subset \R$ be a finite set, let $D=A-A$, and let $\eps>0$ be a real number.
    Then for some constant $C'(\eps)>0$ one has
\begin{equation}\label{f:RN_prod*}
    |DD|,\, |D/D|  \gg_\eps |D| \cdot \min\{ |D|^3 |A|^{-(3+\eps)},  \exp(C'(\eps) \log^{1/3-o(1)} |D|) \} \,.
\end{equation}
\label{c:RN_prod*}
\end{corollary}

Avoiding using  either of the
strong results of Sanders and Schmidt, we prove the following result (see Theorem \ref{t:DD} of section \ref{sec:proof}).

\begin{theorem}
    Let $A\subset \R$ be a finite set and let $D=A-A$.
    Then
\begin{equation}\label{f:RN_prod*_my}
    |D D|,\, |D/D|  \gg |D|^{1+\frac{1}{12}} \log^{-\frac{1}{4}} |D|  \,.
\end{equation}
\label{t:RN_prod*_my}
\end{theorem}

The bound (\ref{f:RN_prod*_my}) can be considered as a new necessary condition for a set to be a difference set of the form $A-A$.
Namely, any such set must have a large product set and a large quotient set.

One might think that the optimal version of (\ref{f:RN_prod*_my}) should state that
$|DD|, |D/D| \gg |D|^{2-\eps}$ for arbitrary $\eps>0$, but this is not true: see Proposition \ref{p:lower_D} which gives examples of sets $A$ with
$|DD|, |D/D| \ll |D|^{3/2}$.

Also, it was conjectured in \cite{RN_Z} that if $|(A+ A) / (A+A)| \ll |A|^2$ or $|(A- A) / (A-A)| \ll |A|^2$, then $|A\pm A| \ll |A|$. The authors obtained some first results in this  direction.
A refined version of Theorem \ref{t:RN_prod*_my},  Theorem \ref{t:DD} below, implies the following result.

\begin{corollary}
    Let $A\subset \R$ be a finite set.
    Suppose that  $|(A- A) / (A-A)| \ll |A|^2$ or  $|(A- A) (A-A)| \ll |A|^2$.
    Then
\begin{equation*}\label{}
    |A-A| \ll |A|^{2-\frac{1}{5}} \log^{\frac{3}{10}} |A| \,.
\end{equation*}
\label{c:RN_intr}
\end{corollary}

\bigskip

A simple consequence of the conjectured bound (\ref{f:problem_A_P}) is that
$A-A \neq P$
for sets $P$ with small product/quotient set.
Our weaker estimate (\ref{f:RN_prod*_my}) gives the same, so, in particular, geometric progressions are not (symmetric) difference sets of the form $A-A$.
An analog of geometric progressions in prime fields $\F_p$
are {\it multiplicative subgroups}.
Obtaining
an appropriate version
of Theorem \ref{t:RN_prod*_my} in the finite-fields setting and using further tools,
we
obtain the following theorem.

\begin{theorem}
Let $p$ be a prime number, let $\G \subset \F_p$ be a multiplicative subgroup, let $|\G| < p^{3/4}$, and let $\xi \neq 0$ be an arbitrary residue.
Suppose that for some $A\subset \F_p$
one has
\begin{equation}\label{cond:A-A_new_intr}
    A-A \subseteq \xi \G \bigsqcup \{ 0 \} \,.
\end{equation}
    Then $|A| \ll |\G|^{4/9}$.
    If $|\G| \ge p^{3/4}$, then $|A| \ll |\G|^{4/3} p^{-2/3}$.
    In particular, for any $\eps>0$ and sufficiently large $\G$, $|\G| \le p^{4/5-\eps}$ we have that
    $$
        A-A \neq \xi \G \bigsqcup \{ 0 \} \,.
    $$
\label{t:A-A_new_intr}
\end{theorem}

Results on representations of sets (and multiplicative subgroups, in particular) as sumsets or product sets can be found in \cite{DS_AD}, \cite{GK}, \cite{Sarkozy_residues}, \cite{Solodkova_S}, \cite{Shparlinski_AD}.
For example, it was proved in \cite{Shparlinski_AD} that any set satisfying (\ref{cond:A-A_new_intr}) has size $|A| \le |\G|^{1/2+o(1)}$.
In \cite{Solodkova_S} the author refined this result in the special case when $A$ is a multiplicative subgroup and obtained  the estimate
$|A| \le |\G|^{1/3+o(1)}$.
In our new  Theorem \ref{t:A-A_new_intr} we have to deal with the general case of arbitrary set $A$ and it is the first result of such type.

\bigskip

The results of the
article
allow us to make
a first tiny tiny step towards answering a beautiful question of P.~Hegarty~\cite{Hegarty_question}.
\medskip

{\bf Problem.}
{\it Let $P \subseteq A+A$ be a strictly convex (concave) set.
Is it necessarily true that $|P| = o (|A|^2)$?
}

\medskip
Recall that a sequence of real numbers $A = \{ a_1 < a_2 < \dots < a_n \}$ is called strictly {\it convex (concave)}
if the consecutive differences $a_i - a_{i-1}$ are strictly increasing (decreasing).
It is known  that from a combinatorial point of view sets $P$ with small product set have behaviour  similar to convex (concave) sets (see \cite{SS1} or discussion before Corollary \ref{c:convex_d}), but in the opinion of the author they have a simpler structure.
%
The following proposition is a consequence of Theorem \ref{t:RN_prod*_my}.

\begin{corollary}
    Let $A \subset \R$ and
    let $D=A-A$.
    Suppose that $|DD| \le M|D|$ or $|D/D| \le M |D|$, where $M\ge 1$.
    Then
\begin{equation*}\label{f:further1_p}
    |A| \ll_M 1 \,.
\end{equation*}
\label{c:M^C_intr}
\end{corollary}

Thus, Corollary \ref{c:M^C_intr} proves the conjecture of Hegarty in the case of pure difference sets $P$ instead of sumsets and where we assume that $P$
has
small product/quotient set instead of assuming convexity.

\bigskip

In the proof we develop some ideas from \cite{Solodkova_S},
combining them with
the
Szemer\'{e}di--Trotter Theorem (see section \ref{sec:SzT}), as well as
with
a new simple combinatorial observation, see formula (\ref{f:R_main}).
The last formula tells  us that if one forms a set $D/D$, $D:=A-A$ (which is known as $Q[A]$ in the literature, see e.g. \cite{TV}), then the set $D/D$ contains a large subset $R \subseteq D/D$ which is {\it additively rich}.
Namely, we consider
$$
    R =  R [A] = \left\{ \frac{a_1-a}{a_2 - a} ~:~ a,a_1,a_2\in A,\, a_2 \neq a \right\} \subseteq D/D
$$
and note that $R = 1-R$.
By the Szemer\'{e}di--Trotter Theorem
the existence of such additive structure in $R$
means that the product of $R$ is large and hence the product of $D$ is large as well.

The paper is organized as follows.
In section \ref{sec:preliminaries} we give a list of the results, which will be further used  in the text.
In the next section we discuss some consequences of the Szemer\'{e}di--Trotter Theorem in its uniform and modern form.
In Section \ref{sec:proof} we prove our main Theorem \ref{t:DD} which implies  Theorem \ref{t:RN_prod*_my} and Corollary \ref{c:RN_intr}.
In the next section we deal with the prime fields case and obtain Theorem \ref{t:A-A_new_intr}
above.
Finally, the constants in Theorem \ref{t:RN_prod*_my} and Corollary \ref{c:RN_intr} can be improved in the case of the quotient set $D/D$ but it requires  much more work -- see section \ref{sec:refined}.
In the appendix we discuss some generalizations of the quantities from section \ref{sec:SzT}.

Let us
conclude with a few comments regarding the notation used in this paper.
All logarithms are to base $2.$ The signs $\ll$ and $\gg$ are the usual Vinogradov symbols.
When the constants in the signs depend on some parameter $M$, we write $\ll_M$ and $\gg_M$.

The author is grateful to D. Zhelezov and S. Konyagin for useful discussions,
and to O. Roche--Newton and M. Rudnev who pointed out to him how to improve Theorems \ref{t:DD}, \ref{t:DD_Fp}.
Also, he  thanks the anonymous referees for a careful reading of the text and for helpful suggestions.

\section{Preliminaries}
\label{sec:preliminaries}

Let $\Gr = (\Gr, +)$ be an abelian group with the group operation $+$.
We
begin with the famous Pl\"{u}nnecke--Ruzsa inequality (see e.g. \cite{TV}).

\begin{lemma}
Let $A,B\subseteq \Gr$ be two finite sets with $|A+B| \le K|A|$.
Then for all positive integers $n,m$ we have the inequality
\begin{equation}\label{f:Plunnecke}
    |nB-mB| \le K^{n+m} |A| \,.
\end{equation}
Furthermore, for any $0< \d < 1$ there exists $X \subseteq A$ such that $|X| \ge (1-\d) |A|$ and such that for any integer $k$ one has
\begin{equation}\label{f:Plunnecke_X}
    |X+kB| \le (K/\d)^k |X| \,.
\end{equation}
\label{l:Plunnecke}
\end{lemma}

We also need Ruzsa's triangle inequality (see, e.g., \cite{TV}).

\begin{lemma}
    Let $A,B,C \subseteq \Gr$ be three finite sets.
    Then
\begin{equation}\label{f:Ruzsa_triangle}
    |C| |A-B| \le |A-C| |B-C| \,.
\end{equation}
\label{l:Ruzsa_triangle}
\end{lemma}



In this paper we have to deal with the quantity $\T (A,B,C,D)$, see \cite{R_Minkovski}, \cite{Solodkova_S}
($\T$ standing for collinear {\it triples})
$$
    \T (A,B,C,D) := \sum_{c \in C,\, d\in D} \E^\times (A-c,B-d)
        =
$$
\begin{equation}\label{f:T_energy}
        =
            |\{ (a-c) (b-d) = (a'-c)(b'-d) ~:~ a,a'\in A,\, b,b' \in B,\, c\in C,\, d\in D \}|  \,.
\end{equation}
If $A=B$ and $C=D$, then we write $\T(A,C)$ for $\T (A,A,C,C)$
and we write $\T(A)$ for $\T (A,A,A,A)$.

\bigskip

In  \cite{J_PhD} Jones proved  a good upper estimate for the quantity $\T(A)$, $A\subset \R$
(another proof was obtained by Roche--Newton in \cite{R_Minkovski}).
Upper bounds for $\T(A)$ when $A$
belongs to a prime field can be found in \cite{AMRS} and \cite{Solodkova_S}.

\begin{theorem}
    Let $A \subset \R$ be finite.
    Then
\begin{equation}\label{f:R_Minkovski}
    \T(A) \ll |A|^4 \log |A| \,.
\end{equation}
\label{t:R_Minkovski}
\end{theorem}

\begin{theorem}
    Let $p$ be a prime number and let $A\subset \F_p$ be a subset with $|A|< p^{2/3}$.
    Then
$$
    \T(A) \ll |A|^{9/2} \,.
$$
\label{t:T(A)_Fp}
\end{theorem}

\bigskip

The method of our  paper
relies on the famous Szemer\'edi--Trotter Theorem
\cite{sz-t}, see also \cite{TV}.
Let us recall the relevant definitions.

We call a set $\mathcal{L}$ of continuous
plane curves a {\it pseudo-line system} if any two members of $\mathcal{L}$
have at most one point in common.
Define the {\it number of incidences} $\mathcal{I} (\mathcal{P},\mathcal{L})$ between points and pseudo--lines  to be
$\mathcal{I}(\mathcal{P},\mathcal{L})=|\{(p,l)\in \mathcal{P}\times \mathcal{L} : p\in l\}|$.

\begin{theorem}\label{t:SzT}
Let $\mathcal{P}$ be a set of points and let $\mathcal{L}$ be a pseudo-line system.
Then
$$\mathcal{I}(\mathcal{P},\mathcal{L}) \ll |\mathcal{P}|^{2/3}|\mathcal{L}|^{2/3}+|\mathcal{P}|+|\mathcal{L}|\,.$$
\end{theorem}

\bigskip

\begin{remark}
    If we redefine a pseudo-line system as a  family of continuous plane curves with $O(1)$ points in common
    and if any two points are simultaneously incident to at most $O(1)$ curves,
    then Theorem \ref{t:SzT} remains true: see e.g. \cite[Theorem 8.10]{TV} for precise bounds in this
    direction.
\end{remark}

\bigskip

Now let us recall the  main result of \cite{sv}.

\begin{theorem}
    Let $\G\subseteq \f_p$ be a multiplicative subgroup,
    let $k\ge 1$ be a positive integer, and let $x_1,\dots,x_k$ be different nonzero elements of $\f_p$.
    Also, let
    $$
        32 k 2^{20k \log (k+1)} \le |\G|\,, \quad  p \ge 4k |\G|  ( |\G|^{\frac{1}{2k+1}} + 1 ) \,.
    $$
    Then
    \begin{equation}\label{f:main_many_shifts}
        |\G\bigcap (\G+x_1) \bigcap \dots (\G+x_k)| \le 4 (k+1) (|\G|^{\frac{1}{2k+1}} + 1)^{k+1} \,.
    \end{equation}
    The same holds if one replaces  $\G$ in (\ref{f:main_many_shifts}) by any cosets of $\G$.
\label{t:main_many_shifts}
\end{theorem}

Thus,
the theorem above
asserts that
$|\G\bigcap (\G+x_1) \bigcap \dots (\G+x_k)| \ll_k |\G|^{\frac{1}{2}+\alpha_k}$,
provided that
$1 \ll_k |\G| \ll_k p^{1-\beta_k}$,
where $\alpha_k, \beta_k$ are some sequences of positive numbers, and $\alpha_k, \beta_k \to 0$ as $k\to \infty$.

\bigskip

We finish this  section with a result from \cite{Solodkova_S}, see Lemma 19.

\begin{lemma}
    Let $\G \subseteq \f_p$ be a multiplicative subgroup, and $k$ be a positive integer.
    Then for any nonzero distinct elements $x_1,\dots,x_k$ of $\f_p$ one has
\begin{equation}\label{f:C_for_subgroups}
    |\G \bigcap (\G+x_1) \bigcap \dots \bigcap (\G+x_k)| = \frac{|\G|^{k+1}}{(p-1)^k} + \theta k 2^{k+3} \sqrt{p} \,,
\end{equation}
    where $|\theta| \le 1$.
\label{l:C_for_subgroups}
\end{lemma}

\section{Some consequences of the Szemer\'{e}di--Trotter Theorem}
\label{sec:SzT}


In this  section
we discuss some
implications
of
Szemer\'{e}di--Trotter Theorem \ref{t:SzT}, which are given in a modern form
(see e.g. \cite{RSS}).
We start with a definition from \cite{s_sumsets}.


\begin{definition}
A finite set $A \subset \R$ is said to be of \emph{Szemer\'{e}di--Trotter type}
(abbreviated as \emph{SzT--type}) if there exists a parameter $D(A)>0$ such that inequality
\begin{equation}\label{f:SzT-type}
\bigl|\bigl\{ s\in A-B ~\mid~ |A\cap (B+s)| \ge \tau \big\}\big|
    \le
        \frac{D (A) |A| |B|^2}{\tau^{3}}\,,
\end{equation}
holds
 for every finite set $B\subset \R$ and every real number $\tau \ge 1$.
\label{def:SzT-type}
\end{definition}

 So, $D(A)$ can be considered as the infimum of numbers such that (\ref{f:SzT-type})
 takes place
 for any $B$ and $\tau \ge 1$
 but, of course, the definition is applicable just for sets $A$ with  small quantity $D(A)$.

\bigskip

Now we can introduce a new characteristic of a set $A\subset \R$, which
can be considered as a generalization
of $D(A)$.

\begin{definition}
Let $\Phi : \R^2 \to \R$ be any function.
A finite set $A \subset \R$ is said to be of \emph{Szemer\'{e}di--Trotter type relative to $\Phi$}
(abbreviated as \emph{SzT$_\Phi$--type}) if there exists a parameter $D_\Phi (A)>0$ such that inequality
\begin{equation}\label{f:SzT-type'}
\bigl|\bigl\{ x  ~:~   | \{ (a,b) \in A \times B,\, a =\Phi(b,x) \} | \ge \tau \big\}\big|
    \le
        \frac{D_\Phi (A) |A| |B|^2}{\tau^{3}}\,,
\end{equation}
holds
 for every finite set $B\subset \R$ and every real number $\tau \ge 1$.
\label{def:SzT-type'}
\end{definition}

\bigskip

The usual quantity $D(A)$ corresponds to the case $\Phi(x,y) = x+y$.

\bigskip

Any SzT--type set has small number of solutions of a wide class of linear equations, see, e.g.,  \cite[Corollary 8]{KS_smd} (where nevertheless another quantity $D(A)$ was used)
and
 \cite[Lemmas 7, 8]{s_sumsets}, say.
This is another illustration of that  fact in a more general context.

\begin{corollary}
    Let $A\subset \R$ be a finite set of SzT$_\Phi$--type.
    For any finite $B\subset \R$ let
$$
    \a_{A,B} (x) = \a^{\Phi}_{A,B} (x) := |\{ (a,b) \in A\times B ~:~ a =\Phi(b,x) \}| \,.
$$
    Then
\begin{equation}\label{f:Phi_cor}
    | \{ x ~:~ \a_{A,B} (x) > 0 \}| \gg |A| |B|^{1/2} D_{\Phi}^{-1/2} (A) \,,
\end{equation}
    and
\begin{equation}\label{f:Phi_cor2}
    \E^{\Phi}_2 (A,B) := \sum_x \a^2_{A,B} (x) \ll D^{1/2}_\Phi (A) \cdot |A| |B|^{3/2} \,.
\end{equation}
\label{c:Phi_cor}
\end{corollary}
\begin{proof}
Let $D=D_{\Phi} (A)$.
Put $S = \{ x ~:~ \a_{A,B} (x) > 0 \}$ and $\tau = |A| |B| /(2|S|)$.
Then
$$
    |A| |B| = \sum_{x\in S} \a_{A,B} (x) \le 2  \sum_{x\in S'} \a_{A,B} (x) \,,
$$
    where $S' = \{ x\in S ~:~ \a_{A,B} (x) \ge \tau \}$.
    By our assumption $A$ is a set of SzT$_\Phi$--type.
    Thus, arranging $\a_{A,B} (x_1) \ge \a_{A,B} (x_2) \ge \dots $, we get
    $\a_{A,B} (x_j) \le (D |A| |B|^2)^{1/3} j^{-1/3}$.
    Whence
$$
    |A| |B| \ll \sum_{j=1}^T (D |A| |B|^2)^{1/3} j^{-1/3} \ll (D |A| |B|^2)^{1/3} T^{2/3} \,,
$$
    where $T^{1/3} = (D |A| |B|^2)^{1/3} \tau^{-1}$.
    It follows that
$$
    |A| |B| \ll D |A| |B|^2 \tau^{-2} \ll D |A|^{-1} |S|^2
$$
    as required.
    Estimate (\ref{f:Phi_cor2}) can be obtained similarly, or see the proof of Lemma 7 in  \cite{s_sumsets}.
    Also notice  that this  bound, combined  with the Cauchy--Schwarz inequality, implies (\ref{f:Phi_cor}).
This concludes the proof.
$\hfill\Box$
\end{proof}

\bigskip

 In \cite{RR-NS} authors considered the quantity
\begin{equation}\label{f:RN}
    \tilde{d}_{+} (A) = \inf_{f} \min_{C} \frac{|f(A) + C|^2}{|A| |C|} \,,
\end{equation}
    where the  infimum is taken over convex/concave functions $f$
    and proved that $D_+ (A) \ll \tilde{d} (A)$.
    In a similar way, one can consider the quantity $D_\times (A)$, which corresponds to $\Phi(x,y) = xy$, and obtain
\begin{equation}\label{f:S_t}
    D_\times (A) \ll \inf_f \min_{C} \frac{|f(A)+C|^2}{|A||C|} \,,
\end{equation}
    where the infimum is taken over all functions $f$ such that for any $b\neq b'$ the function
    $f(xb) - f(xb')$ is monotone.
    The last bound is a generalization of \cite[Lemma 15]{J_R-N}.
    For the rigorous  proof of formula (\ref{f:S_t}) and
    the proof of
    a similar upper bound for the quantity $D_\Phi (A)$, see the appendix.

\section{The proof of the main result}
\label{sec:proof}

First of all let us derive a simple consequence of Theorem \ref{t:SzT}.

\begin{lemma}
    Let $A,B,C,D\subset \R$ be four finite sets.
    Then for any nonzero $\alpha$ one has
\begin{equation}\label{f:2/3}
    |A \cap (B+\a)| \ll (|C| |D|)^{-1/3}
        (|AC||BD|)^{2/3} + |D|^{-1} |BD| +|C|^{-1} |AC|\,.
\end{equation}
    In a dual way
\begin{equation}\label{f:2/3'}
    |A \cap (\a/B) | \ll (|C| |D|)^{-1/3}
        (|A+C||B+D|)^{2/3} + |D|^{-1} |B+D| +|C|^{-1} |A+C|\,.
\end{equation}
\label{l:2/3}
\end{lemma}
\begin{proof}
We can remove the origin from $A,B,C,D$ if we want.
First, note  that
the number of solutions of the equation $a-b = \a$ with $a \in A$ and $b \in B$ is equal to $|A \cap (B+\a)|$.
Hence,
$$
    |A \cap (B+\a)|
        \ll
            (|C||D|)^{-1} |\{ pc^{-1} - p_* d^{-1} = \a ~:~ p \in AC,\, p_* \in BD,\, c \in C,\, d\in D\}| \,.
$$
Let  $\mathcal{P}$ be the set of points  $\{ (c^{-1}, p_*) ~:~ c \in C,\, p_* \in BD \}$  and let
$\mathcal{L}$ be the set of lines $\{ l_{s,t} \}$ where $s\in AC$, $t\in D^{-1}$
and $l_{s,t} = \{ (x,y) ~:~ sx - ty = \a \}$.
Clearly, $|\mathcal{P}| = |C| |BD|$ and $|\mathcal{L}| = |D| |AC|$.
Using Theorem \ref{t:SzT} with the set of points equal to $\mathcal{P}$ and the set of lines equal to $\mathcal{L}$, we get
$$
    |A \cap (B+\a)| \ll (|C||D|)^{-1} ( (|\mathcal{P}| |\mathcal{L}|)^{2/3} + |\mathcal{P}| + |\mathcal{L}|)
        \ll
$$
$$
        \ll
            (|C||D|)^{-1} ( (|C| |D| |AC| |BD| )^{2/3} + |C| |BD| + |D| |AC|)
                \ll
$$
$$
                \ll
                    (|C||D|)^{-1/3} (|AC| |BD|)^{2/3} + |D|^{-1} |BD| +|C|^{-1} |AC| \,.
$$

Similarly, in order to prove (\ref{f:2/3'}), we consider the equation $ab=\a$, $a\in A$, $b\in B$, and also
the equation
$
    (p-c)(p_*-d) = \a
$,
$p\in A+C$, $p_* \in B+D$, which correspond to the curves $l_{p,d} = \{ (x,y) ~:~ (p-x) (y-d) = \a \}$
and the points $\mathcal{P} = C\times (B+D)$.
It is easy to check that any two curves $l_{p,d}$, $l_{p',d'}$ have at  most two points in common.
After that one applies  the Szemer\'{e}di--Trotter Theorem one more time
and performs the calculations above.
This completes  the proof.
$\hfill\Box$
\end{proof}

\bigskip

Take arbitrary finite sets $A,B \subset \R$, $|B|>1$ and
consider the quantity
$$
    R[A,B] = \left\{ \frac{a-b}{b_1 - b} ~:~ a\in A,\, b, b_1 \in B,\, b_1 \neq b \right\}
    \,.
$$
Note  that for any $x$ one has $R[A+x,B+x] = R[A,B]$ and that for all nonzero $\lambda$ one has
$R[\lambda A, \lambda B] = R[A,B]$.

Now let us make a crucial observation about the set $R[A,B]$, which says that the set is somehow additively structured.
Because $\frac{a-b}{b_1 - b} - 1 = -\frac{a-b_1}{b-b_1}$, we have that
\begin{equation}\label{f:R_main}
     R[A,B] = 1-R[A,B] \,.
\end{equation}
Write $R[A]$ for $R[A,A]$
and note  that $R^{-1} [A] = R[A]$.
Geometrically,  the set
$$R[A] = \left\{ \frac{a_2-a}{a_1-a} ~:~ a,a_1,a_2 \in A,\, a_1 \neq a \right\}$$
is the set of all fractions of
(oriented)
lengths of segments $[a_1,a]$, $[a,a_2]$.
With this point of view the formulas $R^{-1} [A] = R[A]$ and $R[A] = 1-R[A]$ become
almost
trivial.
Also note that $0,1 \in R[A]$ and that putting $D=A-A$, we have
$R[A] \subseteq D/D \subseteq R[A] \cdot  R[A]$.
Thus, the products $(D/D)^n$, $n\in \N$ are controlled by $R^m [A]$, $m\in \N$ and vice versa.

\bigskip

{\bf Question.}
Let $A\subset \R$ be a finite set.
Is it true that $|R[A]| \gg |A-A|$?
Similarly, is it true that $|R[A]| \gg |A/A|$?

\bigskip

A simple consequence of Theorem \ref{t:R_Minkovski} is the following
lower bound for the size of the set $R[A]$, see \cite{J_PhD}, \cite{R_Minkovski}.

\begin{theorem}
    Let $A\subset \R$ be a finite set.
    Then
\begin{equation}\label{f:R_size}
    |R[A]| \gg \frac{|A|^2}{\log |A|} \,.
\end{equation}
\label{t:R_size}
\end{theorem}

Now we can prove the main result of this  section.
The author
thanks  Misha Rudnev,  who pointed out to the author how to improve Theorems \ref{t:DD}, \ref{t:DD_Fp} below.

\begin{theorem}
    Let $A\subset \R$ be a finite set and let $D=A-A$.
    Then
\begin{equation}\label{f:DD}
    |DD|,\, |D/D| \gg |D|^{\frac{5}{6}} |R [A]|^{\frac{1}{4}}
        \gg
            |D|^{\frac{5}{6}} |A|^{\frac{1}{2}} \log^{-\frac{1}{4}} |A| \,.
\end{equation}
    In particular,
\begin{equation}\label{f:DD'}
    |DD|,\, |D/D|
        \gg
            |D|^{1+\frac{1}{12}} \log^{-\frac{1}{4}} |D| \,.
\end{equation}
\label{t:DD}
\end{theorem}
\begin{proof}
Put  $R=R[A]$.
Using identity (\ref{f:R_main})
as well as
Lemma \ref{l:2/3} with $A=-R$, $B=R$, $C=D=A-A$ and $\a=-1$, we get
\begin{equation}\label{f:R_D}
    |R| = |-R \cap(R-1)| \ll |D|^{-2/3} |RD|^{4/3} \le |D|^{-2/3} |DD/D|^{4/3} \,.
\end{equation}
Applying Lemma \ref{l:Plunnecke} in its multiplicative form with $A=D$, $B=D$ or $B=D^{-1}$, we have
$$
    |DD/D| \le |D|^{-2} \min\{ |DD|, |D/D| \}^3 \,.
$$
Combining the last two bounds, we obtain the first inequality from (\ref{f:DD}).
The second inequality follows from Theorem \ref{t:R_size}.
The trivial estimates  $|A| \le |D| \le |A|^2$ imply formula (\ref{f:DD'}).
This completes  the proof.
$\hfill\Box$
\end{proof}

\bigskip

From identity (\ref{f:R_main}) and the results of section \ref{sec:SzT}, it follows that the set $R$ has some interesting properties, which we give in the next proposition, one consequence of which is that $R$ has a large product set $|RR| \gg |R|^{\frac{5}{4}}$.

\begin{proposition}
    Let $A\subset \R$ be a finite set, and $R=R[A]$.
    Then $D_\times (R) \ll |RR|^2 / |R|^2$.
    In particular, for any $B\subset \R$
    one has
\begin{equation}\label{f:exp_R}
    |RB| \gg \frac{|R|^2 |B|^{1/2}}{|RR|} \,.
\end{equation}
\label{p:exp_R}
\end{proposition}
\begin{proof}
    Put $f(x) = \ln (x-1)$.
    Then one can check that for any $b\neq b'$ the function $f(xb) - f(xb')$ is monotone.
    Formula (\ref{f:S_t}) gives us that
$$
    D_\times (R) \le \min_{C} \frac{|f(R)+C|^2}{|R||C|} \,.
$$
    Now putting $C=\ln R$ and using identity (\ref{f:R_main}), we deduce that $D_\times (R) \le |RR|^2 / |R|^2$.
    The last bound combined with formula (\ref{f:Phi_cor}) of Corollary \ref{c:Phi_cor} implies (\ref{f:exp_R}).
    This completes the proof.
$\hfill\Box$
\end{proof}

\bigskip

One can obtain the inequality $|RR| \gg |R|^{\frac{5}{4}}$ by taking $B=R$ in the general formula (\ref{f:exp_R}).
Actually, since Proposition \ref{p:exp_R} gives the stronger result $D_\times (R) \ll |RR|^2 / |R|^2$, one can apply methods of \cite{s_sumsets} to derive the inequality $|RR| \gtrsim |R|^{14/9} \cdot D_\times (R)^{-5/9}$ and hence  the inequality $|RR| \gtrsim |R|^{24/19}$.

\bigskip

We finish this section with an example of a set $A$ with a difference set that has small product set and small quotient set.

\begin{proposition}
    For any integer $n\ge 1$ there is a set $A\subset \R$, $|A| = n$ such that  $|DD|, |D/D| \le 25 |D|^{3/2}$,
    where $D=A-A$.
\label{p:lower_D}
\end{proposition}
\begin{proof}
    Let $A = \{ 2,4,\dots,2^{n} \}$.
    Then it is easy to see that $n^2 /2 \le |D| = n^2 -n+1 \le n^2$.
    Furthermore,
$$
    DD = \{ (2^i - 2^j) (2^k - 2^l) ~:~ i,j,k,l\in [n] \} = \{ 2^{j+l} (2^{i-j} -1) (2^{k-l} -1) ~:~ i,j,k,l\in [n] \}
$$
    and hence $|DD| \le (2n)^3 \le (2 \sqrt{2 |D|})^3 \le 25 |D|^{3/2}$.
    Similarly,
$$
    D/D = \{ (2^i - 2^j)/(2^k - 2^l) ~:~ i,j,k,l\in [n] \} = \{ 2^{j-l} (2^{i-j} -1)/(2^{k-l} -1) ~:~ i,j,k,l\in [n] \}
$$
    and again $|D/D| \le 25 |D|^{3/2}$.
    This completes the proof.
$\hfill\Box$
\end{proof}

\section{The finite fields case}
\label{sec:FF}

In this  section $p$ is a prime number.
An analog of Lemma \ref{l:2/3}  in the prime fields setting is the following, see \cite{AMRS}.

\begin{lemma}
    Let $A, B \subset \F_p$ be two subsets with $|A|=|B| < p^{2/3}$.
    Then for any nonzero $\alpha \in \F_p$ one has
\begin{equation}\label{f:2/3_Fp}
    |A \cap (B+\a)| \ll |A|^{-1/2}
        |AB|^{4/3} \,.
\end{equation}
    Moreover, for any $C,D \subset \F_p$ with $|C|,|D| < p^{2/3}$
    we have the inequality
\begin{equation}\label{f:2/3_Fp'}
    |A \cap (B+\a)| \ll (|C||D|)^{-1/4}
        (|AC| |BD|)^{2/3} \,.
\end{equation}
\label{l:2/3_Fp}
\end{lemma}

\noindent {\it Sketch of the proof.}
Actually, bound (\ref{f:2/3_Fp}) was proved  in \cite{AMRS}  in a particular case $A=B$ but since  the proof uses a variant of the Szemer\'{e}di--Trotter Theorem in $\F_p$, namely,
\begin{equation}\label{f:Sz_T_F_p_f}
    \mathcal{I}(\mathcal{P},\mathcal{L}) \ll |\mathcal{P}|^{3/4}|\mathcal{L}|^{2/3}+|\mathcal{L}| + |\mathcal{P}|
\end{equation}
for any set $\mathcal{P}$ of the form $\mathcal{P} = \mathcal{A}\times \mathcal{B}$ with $|\mathcal{B}| \le |\mathcal{A}|<p^{2/3}$
(see \cite{AMRS}),
the arguments of the proof of Lemma \ref{l:2/3} are preserved.
In order to obtain (\ref{f:2/3_Fp'}) we use the same method with $\mathcal{P} = C^{-1} \times D^{-1}$ and
$\mathcal{L} = \{ l_{s,t} \}$, where $s\in AC, t\in BD$ and $l_{s,t}:= \{ (x,y) ~:~ sx-ty=\alpha \}$, so
$|\mathcal{L}| = |AC| |BD|$.
Then one has  $\mathcal{I} (\mathcal{P},\mathcal{L}) \ge |C||D| |A\cap (B+\alpha)|$ -- see the proof of Lemma \ref{l:2/3}.
Let $m = \min \{ |A|, |B|\}$.
Suppose that
\begin{equation}\label{f:fixed1}
        |A \cap (B+\a)| \gg (|C||D|)^{-1/4} (|AC| |BD|)^{2/3}
\end{equation}
because otherwise there is nothing to prove.
It follows that
$$
    (|AC| |BD|)^8 \ll m^{12} (|C||D|)^3 \,.
$$
If the second term in (\ref{f:Sz_T_F_p_f}) dominates, then in view of the last inequality, we obtain
$$
    m^{6}(|C||D|)^{3/2}  \gg (|AC||BD|)^4 \gg (|C||D|)^9 \,.
$$
Thus, $|C| |D| \ll m^{4/5}$, and  from (\ref{f:fixed1}) we have
$$
    m^{6/5} \gg |A \cap (B+\a)| (|C||D|)^{1/4} \gg (|AC| |BD|)^{2/3} \ge (|A||B|)^{2/3} \ge m^{4/3},
$$
which is  a contradiction for large $m$.
Similarly, if the third term in (\ref{f:Sz_T_F_p_f}) is the largest one, then
$$
    (|C||D|)^3 \gg (|AC||BD|)^8 \ge (|C||D|)^8,
$$
which is a contradiction for large $C$, $D$.
In the remaining case, we obtain
$$
    |C||D| |A\cap (B+\alpha)| \ll \mathcal{I} (\mathcal{P},\mathcal{L}) \ll |\mathcal{P}|^{3/4}|\mathcal{L}|^{2/3}
        \le
            (|C| |D|)^{3/4} (|AC||BD|)^{2/3}
$$
as required.

\bigskip

Now let us prove an analog of Theorem \ref{t:DD}.

\begin{theorem}
    Let $A\subset \F_p$ be a finite set and let $D=A-A$.
    Suppose that $|R[A]| \le  cp^{5/9}$, where $c>0$ is an absolute constant.
    Then
\begin{equation}\label{f:DD_Fp}
    |DD|,\, |D/D| \gg |D|^{\frac{19}{24}} |R[A]|^{\frac{1}{4}}
        \gg
            |D|^{\frac{19}{24}} |A|^{\frac{3}{8}} \,.
\end{equation}
\label{t:DD_Fp}
\end{theorem}
\begin{proof}
The arguments are the same as in the proof of Theorem \ref{t:DD}.
Put $R=R[A]$.
Suppose that $|D| < p^{2/3}$.
Then, using identity (\ref{f:R_main})
and Lemma \ref{l:2/3_Fp} (with $A=-R$, $B=R$, $\a = -1$, $C=D=A-A$) instead of Lemma \ref{l:2/3}, we get
\begin{equation}\label{f:R_D_FF}
    |R| = |-R \cap(R-1)| \ll |D|^{-1/2} |RD|^{4/3} \le |D|^{-1/2} |DD/D|^{4/3} \,.
\end{equation}
By Lemma \ref{l:Plunnecke}, we have
$$
    |DD/D| \le |D|^{-2} \min\{ |DD|, |D/D| \}^3 \,.
$$
Combining the last two bounds, we obtain the first inequality from (\ref{f:DD_Fp}).
Because $|A| \le |D| < p^{2/3}$, we see that Theorem \ref{t:T(A)_Fp} implies that $|R[A]| \gg |A|^{3/2}$
and this gives us the second inequality.

Now assume that $|D| \ge p^{2/3}$. We shall show that
\begin{equation}\label{tmp:18.01.2016_1}
    |DD|,\, |D/D| \gg |D|^{\frac{19}{24}} |R|^{\frac{1}{4}} \,.
\end{equation}
If inequality (\ref{tmp:18.01.2016_1}) does not hold for $|DD|$, then
$$
    |D| \le |DD|, |D/D| \ll |D|^{\frac{19}{24}} |R|^{\frac{1}{4}} \,.
$$
Hence,
$$
    |R| \gg |D|^{5/6} \ge p^{5/9},
$$
which is a contradiction.
Again, the second inequality in (\ref{f:DD_Fp}) follows from Theorem \ref{t:T(A)_Fp} and the fact that $|A| \le |R| \ll p^{5/9} < p^{2/3}$.
This completes  the proof.
$\hfill\Box$
\end{proof}

\bigskip

Note that an analog of Proposition \ref{p:exp_R} also holds in $\F_p$ because of an appropriate version of Theorem \ref{t:SzT} (see \cite{AMRS}).
Also note that one can relax the condition $|R[A]| \le  cp^{5/9}$ at the cost of a weaker bound (\ref{f:DD_Fp}).

\bigskip

Now we can obtain Theorem \ref{t:A-A_new_intr}
from the introduction.

\begin{theorem}
Let $\G \subset \F_p$ be a multiplicative subgroup with $|\G| < p^{3/4}$ and let $\xi \neq 0$ be an arbitrary residue.
Suppose that for some $A\subset \F_p$
one has
\begin{equation}\label{cond:A-A_new}
    A-A \subseteq \xi \G \bigsqcup \{ 0 \} \,.
\end{equation}
    Then $|A| \ll |\G|^{4/9}$.
    If $|\G| \ge p^{3/4}$, then $|A| \ll |\G|^{4/3} p^{-2/3}$.
    In particular, for any $\eps>0$ and sufficiently large $\G$, $|\G| \le p^{4/5-\eps}$ we have that
    $$
        A-A \neq \xi \G \bigsqcup \{ 0 \} \,.
    $$
\label{t:A-A_new}
\end{theorem}
\begin{proof}
We can assume that $|A|>1$.
Let $\G_* = \G \sqcup \{ 0 \}$ and let $R=R[A]$.
Then in the light of our condition (\ref{cond:A-A_new}), we have
$$
    R\subseteq (\xi \G \sqcup \{ 0 \} ) / \xi \G = \G_* \,.
$$
In particular, $|R| \le |\G| +1 < p^{3/4}+1$.
We know by (\ref{cond:A-A_new}) that $|A| \le |\G|^{1/2+o(1)} < p^{2/3}$ (see \cite{Shparlinski_AD}, \cite{sv}, \cite{Solodkova_S} or just Theorem \ref{t:main_many_shifts}).
Hence, using Theorem \ref{t:T(A)_Fp}, we obtain that $|R| \gg |A|^{3/2}$.
Applying (\ref{f:R_D_FF}) and the last inequality, we get
$$
    |A|^{3/2} \ll |R| = |-R \cap(R-1)| \le |\G_* \cap (1-\G_*)| \,.
$$
Using Stepanov's method from \cite{K_Tula} or just the case $k=1$ of Theorem \ref{t:main_many_shifts}, we obtain
$$
    |A|^{3/2} \ll |\G_* \cap (1-\G_*)| \le |\G \cap (1-\G)| + 2 \ll |\G|^{2/3}
$$
as required.

Now let us suppose that $|\G| \ge p^{3/4}$.
Using formula (\ref{f:C_for_subgroups}) of Lemma \ref{l:C_for_subgroups} with parameter $k=1$  and the previous calculations, we get
$$
    |A|^{3/2} \ll |\G_* \cap (1-\G_*)| \le |\G \cap (1-\G)| + 2 \ll |\G|^{2} / p
$$
as required.
Finally, it is known that if $A+B = \xi \G$ or $A+B = \xi \G \bigsqcup \{ 0 \}$, then $|A| \sim |B| \sim |\G|^{1/2+o(1)}$, see \cite{Shparlinski_AD}, \cite{sv}, \cite{Solodkova_S}.
Thus, we have $\xi \G \bigsqcup \{ 0 \} \neq A-A$ for sufficiently large $\G$.
This completes  the proof.
$\hfill\Box$
\end{proof}

\section{The case $D/D$}
\label{sec:refined}

    The bound (\ref{f:DD'}) as well as a similar estimate (\ref{f:DD_Fp}) can be improved in the case of the quotient set $D/D$.
    We thank Misha Rudnev and Oliver Roche--Newton who pointed out an idea for how to do this.

    First of all, let us generalize our main identity $R[A]= 1-R[A]$ as well as Theorem \ref{t:R_size}.
    Let $A \subset \R$ be a subset with $|A|>1$ and let $D:=A-A$.
    Also, let $X\subseteq D \setminus \{ 0\}$ be an arbitrary subset such that $X=-X$.
    Let
$$
    R_X [A] = \left\{ \frac{a_2-a}{a_1-a} ~:~ a,a_1,a_2 \in A,\, a_1 - a \in X \right\} \,.
$$
    Then
\begin{equation}\label{f:R_main_X}
    R_X [A] = 1 - R_X [A] \,.
\end{equation}
Note  that $R_X [A] \subseteq D/X$.
Finally, let
$$
    \sigma_X (A) := \sum_{x\in X} |A \cap (A+x)| \,.
$$

Let us derive a consequence of Theorem \ref{t:R_Minkovski}.

\begin{theorem}
    Let $A\subset \R$ be finite and let $X\subseteq (A-A) \setminus \{0\}$ be a set with $X=-X$.
    Then
\begin{equation}\label{f:R_size_X}
    |R_X [A]| \gg \frac{\sigma^2_X (A)}{|A|^{2} \log |A|} \,.
\end{equation}
\label{t:R_size_X}
\end{theorem}
\begin{proof}
    We have
$$
    |A| \sigma_X (A) = \sum_{x\in X} \sum_y |A\cap (A+x) \cap (A+y)|
        = \sum_{\la\in R_X[A]}\, \sum_{x\in X} |A\cap (A+x) \cap (A+\la x)| \,.
$$
    Using the Cauchy--Schwarz inequality, we get
\begin{equation}\label{f:sigma_T}
    |A|^2 \sigma^2_X (A)
        \le
            |R_X [A]| \sum_{\la} \left( \sum_x |A\cap (A+x) \cap (A+\la x)| \right)^2
                =
                    |R_X [A]| \T (A) \,.
\end{equation}
Applying Theorem \ref{t:R_Minkovski}, we obtain the result.
This completes  the proof.
$\hfill\Box$
\end{proof}

\bigskip

Note  that we do not use the fact that $A$ is a subset of the reals to get formula (\ref{f:sigma_T}), but just that $A$ belongs to some field.

\bigskip

Now we obtain the main result of this  section.

\begin{theorem}
    Let $A\subset \R$ be a finite set and let $D=A-A$.
    Then
\begin{equation}\label{f:D/D}
    |D/D| \gg
            |D|^{\frac{3}{4}} |A|^{\frac{3}{4}} \log^{-\frac{5}{8}} |A|
                \gg
                    |D|^{1+\frac{1}{8}} \log^{-\frac{5}{8}} |D| \,.
\end{equation}
    In particular, if $|(A- A) / (A-A)| \ll |A|^2$, then $|A-A| \ll |A|^{2-\frac{1}{3}} \log^{\frac{5}{6}} |A|.$
\label{t:D/D}
\end{theorem}
\begin{proof}
Put $L = \log |A|$.
Let
$$
    D' = \{ x \in D ~:~ |A\cap (A+x)| > |A|^2 / (2|D|) \} \,.
$$
We have $\sigma_{D'} (A) \ge |A|^2 /2$.
Now let
$$
    D'_j = \{ x \in D' ~:~ |A|^2 / (2|D|) \cdot 2^{j-1} < |A\cap (A+x)| \le |A|^2 / (2|D|) \cdot 2^{j} \} \,,
$$
where $j\ge 1$.
By the pigeonhole principle there exists $j \ll L$ such that
\begin{equation}\label{tmp:13.02.2016_1-}
    \sigma_{D'_j} (A) \gg |A|^2/L \,.
\end{equation}
Set $\D = |A|^2 / (2|D|) \cdot 2^{j}$ and redefine $D'$
to be the set
$D'_j$.
Further, instead of using the Pl\"{u}nnecke--Ruzsa inequality (\ref{f:Plunnecke}), choose $\d =1/2$ and apply estimate (\ref{f:Plunnecke_X}) in its multiplicative form with $A=D'$ and $B=D^{-1}$.
    It gives us a set $X\subseteq A$ with $|X| \ge |D'|/2$ such that
\begin{equation}\label{tmp:13.02.2016_1}
    |XBB| = |X/(DD)| = |DD/X| \ll |D/D|^2 |X|/|D|^2 \,.
\end{equation}
    Considering the set $X \cup (-X)$, we see that the last bound holds for this new set with possibly  bigger constants.
    Redefining $X$ to be $X \cup (-X)$ we conclude  that the bound holds for a symmetric set.
    After that, using  the arguments of the proof of Theorem \ref{t:DD}, identity  (\ref{f:R_main_X}), and
     Lemma \ref{l:2/3} with $A=-R_X [A]$, $B=R_X [A]$, $C=D=A-A$ and $\a=-1$,
    we get
    $$
        |R_X [A]| = |R_X[A] \cap (1-R_X [A])| \ll |D|^{-2/3} |R_X [A] \cdot D|^{4/3} \le |D|^{-2/3} |DD/X|^{4/3} \,.
    $$
    Substituting the bound (\ref{tmp:13.02.2016_1}) as well as estimate (\ref{f:R_size_X}) of Theorem \ref{t:R_size_X}, we obtain
\begin{equation}\label{tmp:13.02.2016_3-}
    \frac{\sigma^{3/2}_X (A) |D|^{1/2}}{|A|^{3/2} L^{3/4}} \ll |R_X [A]|^{3/4} |D|^{1/2} \ll  |DD/X|
        \ll
            \frac{|D/D|^2 |X|}{|D|^2} \,.
\end{equation}
We have $\sigma_X (A) \ge \D |X| 2^{-1}$ and hence
\begin{equation}\label{tmp:13.02.2016_3}
    |D/D|^2 \gg |D|^{5/2} \frac{|X|^{1/2} \D^{3/2}}{|A|^{3/2} L^{3/4}} \,.
\end{equation}
In view of inequality (\ref{tmp:13.02.2016_1-}) and the bound $|X| \ge |D'|/2$, we know that
$$
    \D |X| \ge \sigma_X (A) \gg \sigma_{D'} (A) \gg |A|^2/L \,.
$$
Substituting the last estimate into (\ref{tmp:13.02.2016_3}) and recalling that  $\D \gg |A|^2 /|D|$, we obtain, finally,
$$
    |D/D|^2 \gg \frac{|D|^{5/2} \D}{|A|^{1/2} L^{5/4}} \gg \frac{|D|^{3/2} |A|^{3/2}}{L^{5/4}} \,.
$$
Thus, we have
$$
    |D/D| \gg \frac{|A|^{3/4} |D|^{3/4}}{L^{5/8}} \ge \frac{|D|^{9/8}}{L^{5/8}}
$$
as required.
This completes  the proof.
$\hfill\Box$
\end{proof}

\bigskip

The same arguments, combined with the method of proof of Theorem \ref{t:DD_Fp}, give us the following theorem.

\begin{theorem}
    Let $A\subset \F_p$ be a finite set, $D=A-A$.
    Suppose that $|A| \le cp^{10/27} \log^{4/9} |A|$, where $c>0$ is an absolute constant.
    Then
\begin{equation}\label{f:D/D_Fp}
    |D/D|
        \gg
            |D|^{\frac{11}{16}} |A|^{\frac{9}{16}} \log^{-\frac{1}{4}} |A| \,.
\end{equation}
\label{t:D/D_Fp}
\end{theorem}
\begin{proof}
With the notation of the proof of Theorem \ref{t:D/D},
we obtain an analog of inequality (\ref{tmp:13.02.2016_3-}), namely,
$$
    \frac{|X|^{3/2} \D^{3/2} |D|^{3/8}}{|A|^{15/8}}
    \ll
    \frac{\sigma^{3/2}_X (A) |D|^{3/8}}{|A|^{15/8}} \ll |R_X [A]|^{3/4} |D|^{3/8} \ll  |DD/X|
        \ll
            \frac{|D/D|^2 |X|}{|D|^2}
            \,.
$$
Of course one needs  to apply  Theorem \ref{t:T(A)_Fp} instead of Theorem \ref{t:R_Minkovski} in the proof.
Using the last formula we obtain after  some calculations the result in the case $|D| < p^{2/3}$.
For larger sets we suppose that (\ref{f:D/D_Fp}) does not hold
and obtain that
$$
    |D| \le |DD|, |D/D| \ll |D|^{\frac{11}{16}} |A|^{\frac{9}{16}} \log^{-\frac{1}{4}} |A|
$$
and hence that
$$
    p^{10/27} \le |D|^{5/9} \ll |A| \log^{-\frac{4}{9}} |A| \,.
$$
The last inequality contradicts the assumption that $|A| \ll p^{10/27} \log^{4/9} |A|$.
This completes  the proof.
$\hfill\Box$
\end{proof}

\section{Appendix}
\label{sec:appendix}

We finish our paper by studying  the quantity $D_\Phi (A)$ and obtaining the bound (\ref{f:S_t}).

\begin{definition}
For any function $\Phi : \R^2 \to \R$ and any finite set $A\subset\R$, let

\begin{equation}\label{f:d_r}
    d_{\Phi} (A) =
        \inf_F \min_{C} \,
        \frac{|F(A,C)|^2}{|A| |C|} \,,
\end{equation}
where the infimum is taken over all functions $F$ such that\\
$1)~$ for any given $b$ and $s$
the curves
\begin{equation}\label{f:l_{b,s}}
    l_{b,s} := \{ (x,y) ~:~ s=F(\Phi(b,x),y) \}
\end{equation}
form a pseudo-line system, \\
$2)~$ the number of solutions to the  equation $s=F(\Phi(b,x),y)$ is at least two.\\
$3)~$ the set $C$ in (\ref{f:d_r}) is chosen over all nonempty subsets of $\R$ such that
for every $a\in A$ we have that $|F(\{ a\},C)| = |C|$ and that $|F(\t{A},C)| \ge |\t{A}|$ for every $\t{A} \subseteq A$.\\
If there is no such a function $F$,
then
 the quantity $d_{\Phi} (A)$ is not defined.
\label{def:d_Phi}
\end{definition}

Let us give some examples.

\begin{example}
    Suppose that for any $b\neq b'$
    the function $F(\Phi(b,x),y) - F(\Phi(b',x),y)$ does not depend on $y$ and is a monotone function of $x$.
    Then the curves defined in (\ref{f:l_{b,s}}) form a pseudo-line system.
    For example, if $F(x,y) = x+y$, then it is enough to ensure that $\Phi(b,x) - \Phi(b',x)$ is monotone.
    If $F(x,y) = xy$, then the monotonicity of $\Phi(b,x)/\Phi(b',x)$ ensures that
    the curves defined in (\ref{f:l_{b,s}}) form a pseudo-line system.
\end{example}


\begin{example}
    It is easy to see that
    $$
        D_{+} (A) \le \tilde{d}_{+} (A) \,.
    $$
    Indeed, just put $F(x,y) = f(x)+y$ and check that the curves $l_{b,s} = \{ (x,y) ~:~ s = f(b+x) + y \}$ satisfy all the required properties of Definition \ref{def:d_Phi} (e.g., the difference $F(\Phi(b,x),y) - F(\Phi(b',x),y) = f(b+x) - f(b'+x)$ does not depend on $y$ and is monotone).
\label{exm:RN}
\end{example}

\bigskip

Let us show  some simple properties of the quantity $d_{\Phi} (A)$.

\begin{lemma}
    Let $A\subset \R$ be a finite set and let $\Phi : \R^2 \to \R$ be a function such that the quantity $d_\Phi (A)$ is defined.
    Then \\
$\circ$ $1\le d_{\Phi} (A) \le |A|$.\\
$\circ$ For any $A'\subseteq A$ one has $d_{\Phi} (A') \le d_{\Phi} (A) \cdot |A|/|A'|$.
\label{l:d_simple}
\end{lemma}
\begin{proof}
    By the assumption, for any $C$ in the  infimum from (\ref{f:d_r})
    we have that
    $|F(A,C)| \ge \max\{ |A|, |C|\}$.
    Thus,
    $\frac{|F(A,C)|^2}{|A| |C|} \ge 1$, and hence $d_{\Phi} (A) \ge 1$.
    Taking $C$ equal to a one--element set and applying a
    trivial bound
    $|F(A,C)|\le |A|$, we
    obtain that
    $d_{\Phi} (A) \le |A|$.

    Let us prove the second part
    of the lemma.
    Take a set $C$ and a function $F$ such that $d_\Phi (A) \ge \frac{|F(A,C)|^2}{|A||C|} -\eps$,
    where $\eps>0$ is  arbitrary,
    such that
    for any element $a\in A$ we have that $|F(\{ a\},C)| \ge |C|$ and that for all $\t{A} \subseteq A$ the inequality
    $|F(\t{A},C)| \ge |\t{A}|$ holds.
    Clearly, for an arbitrary $a\in A'$ we have again $|F(\{ a\},C)| \ge |C|$ and for all $\t{A} \subseteq A'$ we have the inequality 
    $|F(\t{A},C)| \ge |\t{A}|$.
    Using the trivial inequality $|F(A',C)| \le |F(A,C)|$, we obtain
$$
    d_\Phi (A') \le \frac{|F(A',C)|^2}{|A'||C|} \le \frac{|F(A,C)|^2}{|A'||C|} \le  (d_{\Phi} (A) + \eps) \cdot |A|/|A'|
$$
    as required.
$\hfill\Box$
\end{proof}

\bigskip

If $A=\{ 0 \}$, $F(x,y) = xy$, and $\Phi(x,y) = xy$, say, then it is easy to see that $d_\Phi (A) = 0$.
Thus, we need the property
$|F(\{ a \},C)| \ge |C|$ for every $a \in A$
to obtain $d_{\Phi} (A) \ge 1$.

\bigskip


\bigskip

Now we can prove the main technical statement of this  section.

\begin{proposition}
    Let $A, B\subset \R$ be finite sets and let $\Phi : \R^2 \to \R$ be any function
    such that for any fixed $z$ and $a\in A$ we have the inequality
    $|\{ b\in B ~:~ \Phi (b,z) = a \}| \le M$, where $M>0$ is an absolute constant.
    Then for
    every real number $\tau \ge 1$ one has
\begin{equation}\label{f:F_Phi}
    |\{ x ~:~ | \{ (a,b) \in A \times B,\, a =\Phi(b,x) \} | \ge \tau \}| \ll_M \frac{d_\Phi (A) |A| |B|^2}{\tau^3} \,.
\end{equation}
    In other words, $A$ has SzT$_\Phi$--type with $O_M (d_\Phi (A))$.
\label{p:F_Phi}
\end{proposition}
\begin{proof}
First of all note  that we can assume that $\tau \le |B|$, since otherwise the left--hand side of inequality (\ref{f:F_Phi})
is equal to $0$.
Further, by hypothesis, for any fixed $z$ and any $a\in A$ one has $|\{ b\in B ~:~ \Phi (b,z) = a \}| \le M$
and hence 
\begin{equation}\label{tmp:14.01.2016_2}
    \tau \le \min \{|B|, M|A| \} \,.
\end{equation}

Now take an arbitrary nonempty set $C \subset \R$ and any $x$ from the set in the left--hand side of (\ref{f:F_Phi}).
Then for every  $c\in C$ one has
\begin{equation}\label{f:F_xy}
    F(\Phi(b,x),c) = F(a,c) := s \in F(A,C) \,,
\end{equation}
where $F$ is any function from the infimum in the definition of the quantity $d_\Phi (A)$.
Consider the curves
$l_{b,s} = \{ (x,y) ~:~ s=F(\Phi(b,x),y) \}$,
$b\in B$, $s\in F(A,C)$.
By assumption, the set $\mathcal{L} = \{ l_{b,s} \}$ forms a pseudo-line family.
In particular, $|\mathcal{L}| = |B||F(A,C)|$.
Also, let $\mathcal{P}$ be the set of all intersection points defined  by the curves.
Identity (\ref{f:F_xy}) tells us that the point $(x,c)$ belongs to at least $\tau$ curves $l_{b,s}$.
Hence,
$(x,c)$ belongs to a narrow family
of points
$\mathcal{P}_\tau$ of intersections of at least $\tau$ curves $l_{b,s}$.
Using the Szemer\'{e}di--Trotter Theorem \ref{t:SzT}, we obtain
$$
    |C| \cdot |\{ x ~:~ | (a,b) \in A \times B,\, a =\Phi(b,x) | \ge \tau \}|
        \le
            |\mathcal{P}_\tau|
                \ll
$$
$$
                \ll
                    \frac{|F(A,C)|^2 |B|^2}{\tau^3} + \frac{|F(A,C)||B|}{\tau}
                        \ll_M
                            \frac{|F(A,C)|^2 |B|^2}{\tau^3} \,.
$$
To derive
the last
inequality
we have used estimate (\ref{tmp:14.01.2016_2}), the bound $|F(A,C)| \ge |A|$ and
the trivial inequality
\begin{equation*}\label{tmp:14.01.2016_1}
    \tau^2 \le (\min \{|B|, M|A| \} )^2 \le |B| \cdot M|F(A,C)| \,.
\end{equation*}
This concludes the proof.
$\hfill\Box$
\end{proof}

\begin{example}
    Let $F(x,y) = x+y$ and let $\Phi(x,y) = y \sin x$.
    Then the lines $l_{b,s} = \{ (x,y) ~:~ s=y+x \sin b \}$ form a pseudo-line system.
    Nevertheless, for any fixed $z$ one has $|\Phi^{-1} (\cdot, z)| = + \infty$ or zero.
    So, the pseudo-line condition does not imply any bounds for the size of preimages of the map $\Phi$.
\end{example}

\begin{remark}
    It is easy to see
    that
    it is enough to check the condition $|\{ b ~:~ \Phi (b,z) = a \}| \le M$, $a\in A$
    just for $z$ belonging to the set from the left--hand side of (\ref{f:F_Phi}).
\end{remark}

\bigskip

Now let us derive some consequences of Proposition \ref{p:F_Phi}.
%
Inequality (\ref{f:convex_d}) in the corollary below was obtained in \cite{SS3}: see Lemma 2.6.
Here we give a more systematic proof.
The
corollary
asserts 
that $D_{+} (A) < 4$ for any convex/concave set $A$.
On the other hand, from Proposition \ref{p:F_Phi} it follows that $D_{+} (A) \le M^2$ for any $A$ with $|AA| \le M|A|$ or $|A/A| \le M|A|$.
This
demonstrates some combinatorial
similarity between convex sets and sets with small product/quotient set.

\begin{corollary}
    Let $A\subset \R$ be a finite convex set.
    Then $D_{+} (A) < 4$.
    Furthermore, if $A'\subseteq A$ is a subset of $A$,
    then  for
    an arbitrary
    set $B \subset \R$ one has
\begin{equation}\label{f:convex_d}
    |A'+B| \gg |A'|^{3/2} |B|^{1/2} |A|^{-1/2} \,.
\end{equation}
\label{c:convex_d}
\end{corollary}
\begin{proof}
By assumption, $A=g(I)$, where $I = \{1,2, \dots, |A|\}$ and $g$ is a convex function.
In view of Example \ref{exm:RN}, we have $D_{+} (A) \le \tilde{d}_{+} (A)$.
Moreover, substituting $f=g^{-1}$ in formula (\ref{f:RN}) and taking $C=I$, we obtain
$$
    D_{+} (A) \le \tilde{d}_{+} (A) \le \frac{|I+I|^2}{|A||I|} = \frac{(2|I|-1)^2}{|I|^2} < 4 \,.
$$
Inequality (\ref{f:convex_d}) follows from Corollary \ref{c:Phi_cor} and the second part of Lemma \ref{l:d_simple}.
This completes the proof.
$\hfill\Box$
\end{proof}

\bigskip

To derive the bound (\ref{f:S_t}) from Proposition \ref{p:F_Phi},
just put $F(x,y) = f(x)+y$ and consider the curves
$$
    l_{b,s} = \{ (x,y) ~:~ s=f(xb)+y \} \,.
$$
    These curves have infinite cardinality and form a pseudo-line system.
    Clearly, for any nonempty $A$, $C$ such that $0\notin A,C$ one has
    $|F(\{ a\},C)| = |C|$ for every $a\in A$, and $|F(\t{A},C)| \ge |\t{A}|$ holds for an arbitrary $\t{A} \subseteq A$.
    
    Thus, formula (\ref{f:S_t}) follows.


\bibliographystyle{amsplain}


\begin{dajauthors}
\begin{authorinfo}[ilya]
I.D.~Shkredov\\
Steklov Mathematical Institute,\\
ul. Gubkina, 8, Moscow, Russia, 119991\\
\smallskip
and\\
\smallskip

IITP RAS,  \\
Bolshoy Karetny per. 19, Moscow, Russia, 127994\\
\href{mailto:ilya.shkredov@gmail.com}{ilya.shkredov@gmail.com}

\end{authorinfo}
\end{dajauthors}

\end{document}